\documentclass{article}
\usepackage{amsthm} %proof
\usepackage{amssymb} %\mathbb{}
\usepackage{array} %\newcolumntype

\newtheoremstyle{satz_style}	% name
  {10pt}        % Space above
  {10pt}        % Space below
  {\itshape}    % Body font
  {}            % Indent amount 1
  {\bfseries}   % Theorem head font
  {}		% Punctuation after theorem head
  {.5em}	% Space after theorem head (must not be empty!!! at least " ")
  {\thmname{#1}\thmnumber{ #2}.\thmnote{ #3}}	% Theorem head spec (can be left empty, meaning ¡®normal¡¯)

\newtheoremstyle{satz_style2}	% name
  {10pt}        % Space above
  {10pt}        % Space below
  {\itshape}    % Body font
  {}            % Indent amount 1
  {\bfseries}   % Theorem head font
  {\newline}		% Punctuation after theorem head
  {.5em}	% Space after theorem head (must not be empty!!! at least " ")
  {\thmname{#1}\thmnumber{ #2}.\thmnote{ #3}}	% Theorem head spec (can be left empty, meaning ¡®normal¡¯)

\newtheoremstyle{dfn_style}	% name
  {10pt}	% Space above
  {10pt}	% Space below
  {}		% Body font
  {}		% Indent amount 1
  {\bfseries}	% Theorem head font
  {}		% Punctuation after theorem head
  {.5em}	% Space after theorem head 2
  {\thmname{#1}\thmnumber{ #2}.\thmnote{ #3}}	% Theorem head spec (can be left empty, meaning ¡®normal¡¯)

\theoremstyle{satz_style}
\newtheorem{thm}{Satz}[section]
\newtheorem{proposition}[thm]{Proposition}
\newtheorem{lemma}[thm]{Lemma}

\theoremstyle{satz_style2}
\newtheorem{theorem}[thm]{Theorem}
\newtheorem{proposition2}[thm]{Proposition}

\theoremstyle{dfn_style}
\newtheorem{definition}[thm]{Definition}
\newtheorem{example}[thm]{Example}
\newtheorem{remark}[thm]{Remark}

\def \lw {\mathop{\rm Lw}\nolimits}
\def \lc {\mathop{\rm Lc}\nolimits}
\def \lm {\mathop{\rm Lm}\nolimits}
\def \LLex {\mathop{\rm LLex}\nolimits}

\title{Non-Commutative Gebauer-M\"oller Criteria}
\author{Martin Kreuzer\footnote{{martin.kreuzer@uni-passau.de} (Corresponding author)}\\
Fakult\"at f\"ur Informatik und Mathematik\\
Universit\"at Passau, D-94030 Passau, Germany\\
Xingqiang Xiu\footnote{{xiuxingqiang@gmail.com}}\\
College of Mathematics and Statistics\\
Hainan Normal University, Haikou 571158, P.R. China}
\date{}

\begin{document}
\maketitle

\begin{abstract}
For an efficient implementation of Buchberger's Algorithm, it is essential to avoid the treatment of as many unnecessary critical pairs or obstructions as possible. In the case of the commutative polynomial ring, this is achieved by the
Gebauer-M\"oller criteria. Here we present an adaptation of the Gebauer-M\"oller criteria for non-commutative polynomial
rings, i.e.\ for free associative algebras over fields. The essential idea is to detect unnecessary obstructions using
other obstructions with or without overlap. Experiments show that the new criteria are able to detect almost all
unnecessary obstructions during the execution of Buchberger's procedure.

{\bf Keywords:} Gr\"obner basis, free associative algebra, obstruction, Buchberger procedure

{\bf AMS classification:} 16-08, 20-04, 13P10
\end{abstract}

\section{Introduction}
\label{sec1}

Ever since B. Buchberger's thesis \cite{Bu65}, Gr\"obner bases have become a fundamental tool for computations in commutative algebra and algebraic geometry. The most time-consuming part in Buchberger's Algorithm is the computation of the normal remainder of an S-polynomial corresponding to a critical pair. Therefore a significant amount of energy has been spent on reducing the number of critical pairs which have to be treated. After the discovery of various criteria for discarding critical pairs ahead of time by B. Buchberger and H.M. M\"oller (see \cite{Bu79}, \cite{Bu85} and \cite{Mo85}), this subject found an initial resolution via the \emph{Gebauer-M\"oller installation} presented in \cite{GM88} which offers a good compromise between efficiency and the success rate for detecting unnecessary critical pairs.

A very different picture presents itself for Gr\"obner basis computations for two-sided ideals in non-commutative polynomial rings. The basic Gr\"obner basis theory in this case was described by G.H.\ Bergman (see~\cite{Be78}), T.\ Mora (see \cite{Mo86} and \cite{Mo94}) and others, and obstructions, the non-commutative analogue of critical pairs, were studied in \cite{Mo94}. However, since only a few authors endeavoured to implement efficient versions of Buchberger's Procedure for the non-commutative polynomial ring (i.e. the free associative algebra), the subject of minimizing the number of obstructions which have to be treated has received comparatively little attention, and merely a few rules were developed. For instance, the package
{\tt Plural} of the computer algebra system {\tt Singular} implements a version of the product and the chain criterion,
but not the multiply criterion or the leading word criterion. On the other hand, the system {\tt Magma} appears to be based on a variant of the
F4 Algorithm which does not use criteria for unnecessary obstructions.
For an overview on rules which have been developed see for instance~\cite{Co07}.

In this paper, we present generalizations of the Gebauer-M\"oller criteria for non-commutative polynomials. They cover not only the known cases of useless obstructions discussed in \cite{Mo94}, Lemma 5.11 and \cite{Co07}, but form a complete analogue of the results in the commutative case. One of the key ingredients we use for this purpose is the consideration of obstructions without overlaps. We detect useless obstructions, i.e.\ obstructions that can be represented by other obstructions,
using not only obstructions with overlaps but using also those without overlaps. We show that the consideration of obstructions without overlaps does not increase unnecessary computations, since a Gr\"obner representation is inherent in the S-polynomial of every obstruction without overlaps. Consequently, we reduce the number of obstructions efficiently and obtain a non-commutative version of the Gebauer-M\"oller criteria.

This paper is organised as follows. In Section~\ref{sec2} we recall the basic theory of Gr\"obner bases for
two-sided ideals in non-commutative polynomial rings. In particular, we introduce and study obstructions
(see Definitions~\ref{DefObs} and~\ref{DefNObs}, and Lemmas~\ref{sec2lem7} and~\ref{ProdCrit}),
present the Buchberger Criterion (see Proposition~\ref{BuCrit}), and formulate the Buchberger Procedure
(see Theorem~\ref{BuProc}). The non-commu\-ta\-tive analogues of the Gebauer-M\"oller criteria are developed
in Section~\ref{sec3}. They are based on a careful study of the set of newly constructed obstructions which are
produced during the execution of Buchberger's Procedure. As a result, we are able to formulate
the Non-Commutative Multiply Criterion (see Proposition~\ref{ncMCrit}),
the Non-Commutative Leading Word Criterion (see Proposition~\ref{ncLWCrit})
and the Non-Commutative Backward Criterion (see Proposition~\ref{ncBKCrit}). When we  combine these criteria,
the result is a new Improved Buchberger Procedure~\ref{sec3the14}.

The second author has implemented a version of the Buchberger Procedure for non-commutative polynomial rings
in a package for the computer algebra system {\tt ApCoCoA} which includes the non-commutative Gebauer-M\"oller
criteria developed here (see \cite{Ap10}). In the last section, we present experimental results about
the efficiency of the criteria for some cases of moderately difficult Gr\"obner basis computations.

Unless mentioned otherwise, we adhere to the definitions and terminology given in \cite{KR00} and \cite{KR05}.

\section{Gr\"obner Bases in $K\langle X\rangle$}
\label{sec2}

In the following we let $X=\{x_1,\dots,x_n\}$ be a finite set of indeterminates (or a finite alphabet), and $\langle X\rangle$ the monoid of all \emph{words} (or \emph{terms}) $x_{i_1}\cdots x_{i_l}$ where the multiplication is concatenation of words. The empty word will be denoted by $1$. Furthermore, let $K$ be a field, and let
$$K\langle X\rangle=\{c_1w_1+\cdots+c_sw_s\ |\ s\in\mathbb{N},c_i\in K\setminus\{0\},w_i\in\langle X\rangle\}$$
be the non-commutative polynomial ring generated by $X$ over $K$ (or the free associative $K$-algebra generated by $X$). We introduce basic notions of Gr\"obner basis theory in this setting.

\begin{definition}
A \emph{word ordering} on $\langle X\rangle$ is a well-ordering ${\sigma}$ which is compatible with multiplication, i.e. $w_1\geq_{\sigma} w_2$ implies $w_3w_1w_4\geq_{\sigma} w_3w_2w_4$ for all words $w_1,w_2,w_3,w_4\in\langle X\rangle$.
\end{definition}

In the commutative case, a word ordering is usually called a \emph{term ordering} or \emph{monomial ordering}. For instance, the \emph{length-lexicographic ordering} $\LLex$ is a word ordering. It first compares the length of two words and then breaks ties using the non-commutative lexicographic ordering with respect to $x_1>_{\LLex}\cdots>_{\LLex}x_n$. Note that the non-commutative lexicographic ordering by itself is not a word ordering, since it is neither a well-ordering nor compatible with multiplication.

\begin{definition}
Let ${\sigma}$ be a word ordering on $\langle X\rangle$.
\begin{itemize}
\item[(a)] Given a polynomial $f\in K\langle X\rangle\setminus\{0\}$, there exists a unique representation $f=c_1w_1+\cdots+c_sw_s$ with $c_1,\dots,c_s\in K\setminus\{0\}$ and $w_1,\dots,w_s\in\langle X\rangle$ such that $w_1>_{\sigma}\cdots>_{\sigma} w_s$. The word $\lw_{\sigma}(f)=w_1$ is called the \emph{leading word} of $f$ with respect to ${\sigma}$. The element $\lc_{\sigma}(f)=c_1$ is called the \emph{leading coefficient}. We let $\lm_{\sigma}(f)=c_1w_1$ and call it the \emph{leading monomial} of $f$.

\item[(b)] Let $I\subseteq K\langle X\rangle$ be a two-sided ideal. The set $\lw_{\sigma}\{I\}=\{\lw_{\sigma}(f)\ |\ f\in I\setminus\{0\}\}\subseteq\langle X\rangle$ is called the \emph{leading word set} of $I$. The two-sided ideal $\lw_{\sigma}(I)=\langle\lw_{\sigma}(f)\ |\ f\in I\setminus\{0\}\rangle\subseteq K\langle X\rangle$ is called the \emph{leading word ideal} of $I$.

\item[(c)] A subset $G$ of a two-sided ideal $I\subseteq K\langle X\rangle$ is called a \emph{${\sigma}$-Gr\"obner basis} of $I$ if the set of the leading words $\lw_{\sigma}\{G\}=\{\lw_{\sigma}(f)\ |\ f\in G\setminus\{0\}\}$ generates the leading word ideal $\lw_{\sigma}(I)$.
\end{itemize}
\end{definition}

In the following we focus on computations of Gr\"obner bases for two-sided ideals in $K\langle X\rangle$. For readers who want to know further properties and applications of non-commutative Gr\"obner bases, we refer to \cite{Mo94} and \cite{Xiu12}. Throughout this paper we assume that ${\sigma}$ is a word ordering on $\langle X\rangle$. The next algorithm is a central part of all Gr\"obner basis computations.

\begin{theorem}[(The Division Algorithm)]\label{DivAlg}
Let $f\in K\langle X\rangle$, $s\geq 1$, and $G=\{g_{1},\dots,g_{s}\}\subseteq K\langle X\rangle\setminus\{0\}$. Consider the following sequence of instructions.
\begin{itemize}
\item[{\rm (D1)}] Let $k_{1}=\cdots=k_{s}=0, p=0$, and $v=f$.

\item[{\rm (D2)}] Find the smallest index $i\in\{1,\dots,s\}$ such that $\lw_{\sigma}(v)=w\lw_{\sigma}(g_{i})w'$ for some words $w,w'\in\langle X\rangle$. If such an $i$ exists, increase $k_{i}$ by $1$, set $c_{ik_{i}}=\frac{\lc_{\sigma}(v)}{\lc_{\sigma}(g_{i})}, w_{ik_{i}}=w, w'_{ik_{i}}=w'$, and replace $v$ by $v-c_{ik_{i}}w_{ik_{i}}g_{i}w'_{ik_{i}}$.

\item[{\rm (D3)}] Repeat step {\rm (D2)} until there is no more $i\in\{1,\dots,s\}$ such that $\lw_{\sigma}(v)$ is a multiple of $\lw_{\sigma}(g_{i})$. If now $v\neq 0$, then replace $p$ by $p+\lm_{\sigma}(v)$ and $v$ by $v-\lm_{\sigma}(v)$, continue with step {\rm (D2)}.

\item[{\rm (D4)}] Return the tuples $(c_{11},w_{11},w'_{11}),\dots,(c_{sk_{s}},w_{sk_{s}},w'_{sk_{s}})$ and~$p$.
\end{itemize}
This is an algorithm which returns tuples $(c_{11},w_{11},w'_{11}),\dots,(c_{sk_{s}},w_{sk_{s}},w'_{sk_{s}})$ and a polynomial $p\in K\langle X\rangle$ such that the following conditions are satisfied.
\begin{itemize}
\item[{\rm (a)}] We have $f=\sum^{s}_{i=1}\sum^{k_{i}}_{j=1}c_{ij}w_{ij}g_{i}w'_{ij}+p$.

\item[{\rm (b)}] No element of ${\rm Supp}(p)$ is contained in $\langle\lw_{\sigma}(g_{1}),\dots,\lw_{\sigma}(g_{s})\rangle$.

\item[{\rm (c)}] For all $i\in\{1,\dots,s\}$ and all $j\in\{1,\dots,k_{i}\}$, we have $\lw_{\sigma}(w_{ij}g_{i}w'_{ij})\leq_{\sigma}\lw_{\sigma}(f)$. If $p\neq 0$, we have $\lw_{\sigma}(p)\leq_{\sigma}\lw_{\sigma}(f)$.

\item[{\rm (d)}] For all $i\in\{1,\dots, s\}$ and all $j\in\{1,\dots,k_{i}\}$, we have $\lw_{\sigma}(w_{ij}g_{i}w'_{ij})\notin\langle\lw_{\sigma}(g_{1}),\dots,\lw_{\sigma}(g_{i-1})\rangle$.
\end{itemize}
\end{theorem}

Note that the resulting tuples $(c_{11},w_{11},w'_{11}),\dots,(c_{sk_{s}},w_{sk_{s}},w'_{sk_{s}})$ and polynomial $p$ satisfying conditions (a)-(d) are \emph{not} unique. This is due to the fact that in step (D2) of the Division Algorithm there might exist more that one pair $(w,w')$ satisfying $\lw_{\sigma}(v)=w\lw_{\sigma}(g_{i})w'$ (see \cite{Xiu12}, Example 3.2.2). A polynomial $p\in K\langle X\rangle$ obtained in Theorem~\ref{DivAlg} is called a \emph{normal remainder} of $f$ with respect to $G$ and is denoted by ${\rm NR}_{\sigma,G}(f)$.

For $s\geq 1$, we let $F_s=(K\langle X\rangle\otimes_{K}K\langle X\rangle)^s$ be the free two-sided $K\langle X\rangle$-module of rank $s$ with the canonical basis $\{e_{1},\dots,e_{s}\}$, where $e_{i}=(0,\dots,0,1\otimes1,$ $0,\dots,0)$ with $1\otimes1$ occurring in the $i^{\rm th}$ position for $i=1,\dots,s$, and we let $\mathbb{T}(F_{s})$ be the set of terms in $F_{s}$, i.e. $\mathbb{T}(F_{s})=\{we_{i}w'\ |\ i\in\{1,\dots,s\}, w,w'\in\langle X\rangle\}$.

\begin{definition}\label{DefObs}
Let $G=\{g_1,\dots,g_s\}\subseteq K\langle X\rangle\setminus\{0\}$ with $s\geq 1$, and let $i,j\in\{1,\dots,s\}$.
\begin{itemize}
\item[(a)] If there exist some words $w_{i},w'_{i},w_{j},w'_{j}\in\langle X\rangle$ such that $w_{i}\lw_{\sigma}(g_{i})w'_{i}=w_{j}\lw_{\sigma}(g_{j})w'_{j}$, then we call the element
$${\rm o}_{i,j}(w_{i},w'_{i};w_{j},w'_{j})=\frac{1}{\lc_{\sigma}(g_{i})}w_{i}e_{i}w'_{i}-\frac{1}{\lc_{\sigma}(g_{j})}w_{j}e_{j}w'_{j}\in F_{s}$$
an \emph{obstruction} of $g_i$ and $g_j$ whenever it is non-zero. If $i=j$, it is called a \emph{self obstruction} of $g_{i}$. We will denote the \emph{set of all obstructions} of $g_{i}$ and $g_{j}$ by ${\rm Obs}(i,j)$.

\item[(b)] Let ${\rm o}_{i,j}(w_{i},w'_{i};w_{j},w'_{j})\in {\rm Obs}(i,j)$ be an obstruction of $g_{i}$ and $g_{j}$. The polynomial
$$S_{i,j}(w_{i},w'_{i};w_{j},w'_{j})=\frac{1}{\lc_{\sigma}(g_{i})}w_{i}g_{i}w'_{i}-\frac{1}{\lc_{\sigma}(g_{j})}w_{j}g_{j}w'_{j}\in K\langle X\rangle$$
is called the \emph{S-polynomial} of ${\rm o}_{i,j}(w_{i},w'_{i};w_{j},w'_{j})$.
\end{itemize}
\end{definition}

By definition, we have ${\rm o}_{j,i}(w_{j},w'_{j};w_{i},w'_{i})=-{\rm o}_{i,j}(w_{i},w'_{i};w_{j},w'_{j})$ and hence $S_{j,i}(w_{j},w'_{j};w_{i},w'_{i})=-S_{i,j}(w_{i},w'_{i};w_{j},w'_{j})$.

\begin{example}\label{ExObs}
Consider the non-commutative polynomial ring $\mathbb{Q}\langle x,y\rangle$ equipped with the word ordering $\LLex$ on $\langle x,y\rangle$ such that $x>_{\LLex}y$. Let $g_{1}=2x^{2}+yx$ and $g_{2}=xy+x$. Then $\lw_{\LLex}(g_{1})=x^2$ and $\lw_{\LLex}(g_{2})=xy$. We list some parts of self obstructions of $g_1$, obstructions of $g_1$ and $g_2$, and self obstructions of $g_2$ as follows.\\
Self obstructions of $g_1$:
\begin{eqnarray*}
&&\{\pm{\rm o}_{1,1}(w_1,xw_2;w_1x,w_2)\ |\ w_1,w_2\in\langle x,y\rangle\}\\
&&\cup\{\pm{\rm o}_{1,1}(1,w_3x^2;x^2w_3,1)\ |\ w_3\in\langle x,y\rangle\}
\end{eqnarray*}
Obstructions of $g_1$ and $g_2$:
\begin{eqnarray*}
&&\{{\rm o}_{1,2}(w_1,yw_2;w_1x,w_2)\ |\ w_1,w_2\in\langle x,y\rangle\}\\
&&\cup\{{\rm o}_{1,2}(1,w_3xy;x^2w_3,1)\ |\ w_3\in\langle x,y\rangle\}\\
&&\cup\{{\rm o}_{1,2}(xyw_4,1;1,w_4x^2)\ |\ w_3\in\langle x,y\rangle\}
\end{eqnarray*}
Self obstructions of $g_2$:
\begin{eqnarray*}
\{\pm{\rm o}_{2,2}(1,w_1xy;xyw_1,1)\ |\ w_1\in\langle x,y\rangle\}
\end{eqnarray*}
\end{example}

Using Definition \ref{DefObs}, we can characterize Gr\"obner bases in the following way.

\begin{proposition}\label{PropGB}
Let $G=\{g_1,\dots,g_s\}\subseteq K\langle X\rangle\setminus\{0\}$ be a set of polynomials which generate a two-sided ideal $I=\langle G\rangle\subseteq K\langle X\rangle$. Then the following conditions are equivalent.
\begin{itemize}
\item[{\rm (a)}] The set $G$ is a ${\sigma}$-Gr\"obner basis of $I$.

\item[{\rm (b)}] For every obstruction ${\rm o}_{i,j}(w_{i},w'_{i};w_{j},w'_{j})$ in the set $\bigcup_{1\leq i\leq j\leq s}{\rm Obs}(i,j)$, its S-polynomial $S_{i,j}(w_{i},w'_{i};w_{j},w'_{j})$ has a representation
$$S_{i,j}(w_{i},w'_{i};w_{j},w'_{j})=\sum^{\mu}_{k=1}c_{k}w_{k}g_{i_{k}}w'_{k}$$
with $c_{k}\in K, w_{k},w'_{k}\in\langle X\rangle$, and $g_{i_{k}}\in G$ for all $k\in\{1,\dots,\mu\}$ such that $\lw_{\sigma}(w_{j}g_{j}w'_{j})>_{\sigma}\lw_{\sigma}(w_{k}g_{i_{k}}w'_{k})$ if $c_k\neq 0$ for some $k\in\{1,\dots,\mu\}$.
\end{itemize}
\end{proposition}
\begin{proof}
See \cite{Mo94}, Subsection 5.3.
\end{proof}

A presentation of $S_{i,j}(w_{i},w'_{i};w_{j},w'_{j})$ as in Proposition~\ref{PropGB}.b is called a \emph{Gr\"obner representation} of $S_{i,j}(w_{i},w'_{i};w_{j},w'_{j})$ in terms of $G$.

Observe that there are infinitely many obstructions in each set ${\rm Obs}(i,j)$, due to the following two types of \emph{trivial} obstructions.

\begin{itemize}
\item[(T1)] If ${\rm o}_{i,j}(w_{i},w'_{i};w_{j},w'_{j})\in {\rm Obs}(i,j)$, then, for all $w,w'\in\langle X\rangle$ such that $ww'\neq 1$, we have ${\rm o}_{i,j}(ww_{i},w'_{i}w';ww_{j},w'_{j}w')\in{\rm Obs}(i,j)$. In this case, we say that ${\rm o}_{i,j}(ww_{i},w'_{i}w';ww_{j},w'_{j}w')$ is a \emph{proper multiple} of ${\rm o}_{i,j}(w_{i},w'_{i};w_{j},w'_{j})$.

\item[(T2)] For all $w\in\langle X\rangle$, we have ${\rm o}_{i,j}(\lw_{\sigma}(g_{j})w,1;1,w\lw_{\sigma}(g_{i})), {\rm o}_{i,j}(1, w\lw_{\sigma}(g_{j});$ $\lw_{\sigma}(g_{i})w,1)\in{\rm Obs}(i,j)$.
\end{itemize}

Before going on, let us get rid of these two types of trivial obstructions. The following lemma handles trivial obstructions of type (T1).

\begin{lemma}\label{sec2lem7}
If the S-polynomial of ${\rm o}_{i,j}(w_{i},w'_{i};w_{j},w'_{j})\in {\rm Obs}(i,j)$ has a Gr\"ob\-ner representation
in terms of $G$, then, for all $w,w'\in\langle X\rangle$ such that $ww'\neq 1$, the S-polynomial of
${\rm o}_{i,j}(ww_{i},w'_{i}w';ww_{j},w'_{j}w')$ also has a Gr\"obner representation in terms of $G$.
\end{lemma}
\begin{proof}
Without loss of generality, we assume that $S_{i,j}(w_{i},w'_{i};w_{j},w'_{j})$ is non-zero. We write $S_{i,j}(w_{i},w'_{i};w_{j},w'_{j})=\sum^{\mu}_{k=1}c_{k}w_{k}g_{i_{k}}w'_{k}$, where
$c_{k}\in K\setminus\{0\}$, $w_{k},w'_{k}\in\langle X\rangle$, and $g_{i_{k}}\in G$ such that
$\lw_{\sigma}(w_{j}g_{j}w'_{j})>_{\sigma}\lw_{\sigma}(w_{k}g_{i_{k}}w'_{k})$ for all $k\in\{1,\dots,\mu\}$.
For all $w,w'\in\langle X\rangle$, it is clear that $S_{i,j}(ww_{i},w'_{i}w';ww_{j},w'_{j}w')$
$=\sum^{\mu}_{k=1}c_{k}ww_{k}g_{i_{k}}w'_{k}w'$. Since the word ordering $\sigma$ is compatible with multiplication,
we have $w\lw_{\sigma}(w_{j}g_{j}w'_{j})w'>_{\sigma}w\lw_{\sigma}(w_{k}g_{i_{k}}w'_{k})w'$ for all $k\in\{1,\dots,\mu\}$.
Hence we have $\lw_{\sigma}(ww_{j}g_{j}w'_{j}w')>_{\sigma}\lw_{\sigma}(ww_{k}g_{i_{k}}w'_{k}w')$ for all $k\in\{1,\dots,\mu\}$ and $S_{i,j}(ww_{i},w'_{i}w';ww_{j},w'_{j}w')=\sum^{\mu}_{k=1}c_{k}ww_{k}g_{i_{k}}w'_{k}w'$ is a Gr\"obner representation in terms of $G$.
%See \cite{Co07}, Lemma 1.1.
\end{proof}

To deal with trivial obstructions of type (T2), we introduce some terminology as follows.

\begin{definition}
Let $G=\{g_1,\dots,g_s\}\subseteq K\langle X\rangle\setminus\{0\}$ with $s\geq 1$.
\begin{itemize}
\item[(a)]
Let $w_{1},w_{2}\in\langle X\rangle$ be two words. If there exist some words $w,w',w''\in\langle X\rangle$ and $w\neq 1$ such that $w_{1}=w'w$ and $w_{2}=ww''$, or $w_{1}=ww'$ and $w_{2}=w''w$, or $w_{1}=w$ and $w_{2}=w'ww''$, or $w_{1}=w'ww''$ and $w_{2}=w$, then we say $w_{1}$ and $w_{2}$ have an \emph{overlap} at $w$. Otherwise, we say that $w_{1}$ and $w_{2}$ have \emph{no overlap}.
\item[(b)]
Let ${\rm o}_{i,j}(w_{i},w'_{i};w_{j},w'_{j})\in {\rm Obs}(i,j)$ be an obstruction. If $\lw_{\sigma}(g_{i})$ and $\lw_{\sigma}(g_{j})$ overlap at $w\in\langle X\rangle\setminus\{1\}$ in the word $w_{i}\lw_{\sigma}(g_{i})w'_{i}$, then we say that ${\rm o}_{i,j}(w_{i},w'_{i};w_{j},w'_{j})$ has an \emph{overlap} at $w$. Otherwise, we say that ${\rm o}_{i,j}(w_{i},w'_{i};w_{j},w'_{j})$ has \emph{no overlap}.
\end{itemize}
\end{definition}

\begin{example}
Consider Example \ref{ExObs} again.
\begin{itemize}
\item[{\rm (a)}]
The word $x^2$ has an overlap at $x$ with itself. The obstruction ${\rm o}_{1,1}(1,x;x,1)$ has an overlap at $x$ since $\lw_{\LLex}(g_1)=x^2$ overlaps at $x$ with itself in $1\cdot\lw_{\LLex}(g_1)\cdot x=x^3$. The obstruction ${\rm o}_{1,1}(1,w_3x^2;x^2w_3,1)$ with $w_3 $ in $\langle x,y\rangle$ has no overlap since $\lw_{\LLex}(g_1)$ does not overlap with itself in $1\cdot\lw_{\LLex}(g_1)\cdot w_3x^2=x^2w_3x^2$.
\item[{\rm (b)}]
The words $x^2$ and $xy$ have an overlap at $x$. The obstruction ${\rm o}_{1,2}(1,y;x,1)$ has an overlap at $x$ since $\lw_{\LLex}(g_1)=x^2$ and $\lw_{\LLex}(g_2)=xy$ overlap at $x$ in $1\cdot\lw_{\LLex}(g_1)\cdot y=x^2y$. The obstruction ${\rm o}_{1,2}(1,w_3xy;x^2w_3,1)$ with $w_3\in\langle x,y\rangle$ has no overlap since $\lw_{\LLex}(g_1)$ and $\lw_{\LLex}(g_2)$ does not overlap in $1\cdot\lw_{\LLex}(g_1)\cdot w_3xy=x^2w_3xy$. Similarly, the obstruction ${\rm o}_{1,2}(xyw_4,1;1,w_4x^2)$ with $w_4\in\langle x,y\rangle$ has no overlap either.
\item[{\rm (c)}]
The word $xy$ does not have any overlap with itself. Hence, the obstruction ${\rm o}_{2,2}(1,w_1xy;xyw_1,1)$ for any $w_1\in\langle x,y\rangle$ has no overlap.
\end{itemize}
\end{example}

Thus, as shown in (T2), there are infinitely many obstructions without overlaps in each ${\rm Obs}(i,j)$. The following lemma gets rid of these trivial obstructions.

\begin{lemma}\label{ProdCrit}
If ${\rm o}_{i,j}(w_{i},w'_{i};w_{j},w'_{j})\in {\rm Obs}(i,j)$ has no overlap, then $S_{i,j}(w_{i},w'_{i};$ $w_{j},w'_{j})$ has a Gr\"obner representation in terms of $G$.
\end{lemma}
\begin{proof}
See \cite{Mo94}, Lemma 5.4.
\end{proof}

Observe that Lemma~\ref{ProdCrit} is indeed a non-commutative version of the \emph{product criterion} (or \emph{criterion 2}) of Buchberger (cf. \cite{Bu85}).

\begin{definition}\label{DefNObs}
Let $G=\{g_1,\dots,g_s\}\subseteq K\langle X\rangle\setminus\{0\}$ with $s\geq 1$.
\begin{itemize}
\item[(a)] Let $i,j\in\{1,\dots,s\}$ and $i\neq j$. An obstruction in ${\rm Obs}(i,j)$ is called \emph{non-trivial} if it has an overlap and is of the form ${\rm o}_{i,j}(w_{i},1;1,w'_{j})$, or ${\rm o}_{i,j}(1,w'_{i};$ $w_{j},1)$, or ${\rm o}_{i,j}(w_{i},w'_{i};1,1)$, or ${\rm o}_{i,j}(1,1;w_{j},w'_{j})$ with $w_{i},w'_{i},w_{j},w'_{j}\in\langle X\rangle$.

\item[(b)] Let $i\in\{1,\dots,s\}$. A self obstruction in ${\rm Obs}(i,i)$ is called \emph{non-trivial} if it has an overlap and is of the form ${\rm o}_{i,i}(1,w'_{i};w_{i},1)$ with $w_{i},w'_{i}\in\langle X\rangle\setminus\{1\}$.

\item[(c)] Let $i,j\in\{1,\dots,s\}$. The \emph{set of all non-trivial obstructions} of $g_{i}$ and $g_{j}$ will be denoted by ${\rm NTObs}(i,j)$.
\end{itemize}
\end{definition}

Note that in Definition \ref{DefNObs}.b we only consider the form ${\rm o}_{i,i}(1,w'_{i};w_{i},1)$ due to the reason that ${\rm o}_{i,i}(w_{i},1;1,w'_{i})=-{\rm o}_{i,i}(1,w'_{i};w_{i},1)$.

\begin{example}
Consider Example \ref{ExObs} again. The non-trivial obstructions are as follows.
\begin{eqnarray*}
{\rm NTObs}(1,1)&=&\{{\rm o}_{1,1}(1,x;x,1)\}\\
{\rm NTObs}(1,2)&=&\{{\rm o}_{1,2}(1,y;x,1)\}\\
{\rm NTObs}(2,2)&=&\emptyset
\end{eqnarray*}
\end{example}

In the literature, a non-trivial obstruction of the form ${\rm o}_{i,j}(w_{i},1;1,w'_{j})$ is called a \emph{left obstruction}, a non-trivial obstruction of the form ${\rm o}_{i,j}(1,w'_{i};w_{j},1)$ is called a \emph{right obstruction}, and a non-trivial obstruction of the form ${\rm o}_{i,j}(w_{i},w'_{i};$ $1,1)$ or ${\rm o}_{i,j}(1,1;w_{j},w'_{j})$ is called a \emph{center obstruction}. We picture four types of obstructions as follows.

\begin{center}
\newcolumntype{C}{>{\centering\arraybackslash}m{0.5cm}<{}}
\begin{tabular}{*{3}{C}}
\hline
\multicolumn{1}{|c|}{$w_{i}$} & \multicolumn{2}{c|}{$\lw_{\sigma}(g_i)$}\\ \hline\hline
\multicolumn{2}{|c|}{$\lw_{\sigma}(g_j)$} & \multicolumn{1}{c|}{$w'_{j}$}\\ \hline
&\multicolumn{1}{c}{left   obstruction}&\\\hline
\multicolumn{1}{|c|}{$w_{i}$} & \multicolumn{1}{c|}{$\lw_{\sigma}(g_i)$} & \multicolumn{1}{c|}{$w'_{i}$}\\ \hline\hline
\multicolumn{3}{|c|}{$\lw_{\sigma}(g_j)$}\\ \hline
&\multicolumn{1}{c}{center obstruction}&
\end{tabular}\qquad
\begin{tabular}{*{3}{C}}
\hline
\multicolumn{2}{|c|}{$\lw_{\sigma}(g_i)$} & \multicolumn{1}{c|}{$w'_{i}$} \\ \hline\hline
\multicolumn{1}{|c|}{$w_{j}$} & \multicolumn{2}{c|}{$\lw_{\sigma}(g_j)$} \\ \hline
&\multicolumn{1}{c}{right  obstruction}&\\\hline
\multicolumn{3}{|c|}{$\lw_{\sigma}(g_i)$}\\\hline\hline
\multicolumn{1}{|c|}{$w_{j}$} & \multicolumn{1}{c|}{$\lw_{\sigma}(g_j)$} & \multicolumn{1}{c|}{$w'_{j}$}\\\hline
&\multicolumn{1}{c}{center obstruction}&
\end{tabular}
\end{center}

For the sake of convenience, in the  rest of the paper a \emph{trivial obstruction} we mean an obstruction without overlap or a proper multiple of a non-trivial obstruction.

At this point we can refine the characterization of Gr\"obner bases given in Proposition~\ref{PropGB} in the following way.

\begin{proposition2}[(Buchberger Criterion)]\label{BuCrit}
Let $G=\{g_1,\dots,g_s\}\subseteq K\langle X\rangle$ be a set of non-zero polynomials which generate a two-sided ideal $I=\langle G\rangle\subseteq K\langle X\rangle$. Then the set $G$ is a ${\sigma}$-Gr\"obner basis of $I$ if and only if, for each non-trivial obstruction ${\rm o}_{i,j}(w_{i},w'_{i};w_{j},w'_{j})\in\bigcup_{1\leq i\leq j\leq s}{\rm NTObs}(i,j)$, its S-polynomial $S_{i,j}(w_{i},w'_{i};w_{j},w'_{j})$ has a Gr\"obner representation in terms of $G$.
\end{proposition2}
\begin{proof}
This follows directly from Proposition~\ref{PropGB} and Lemmas~\ref{sec2lem7} and~\ref{ProdCrit}. In view of Lemma~\ref{ProdCrit}, it suffices to consider each obstruction with overlap, which is either a non-trivial obstruction or a proper multiple of a non-trivial obstruction. Further, Lemma~\ref{sec2lem7} treats a proper multiple of a non-trivial obstruction via the corresponding non-trivial one. Therefore, it is sufficient to consider only non-trivial obstructions.
\end{proof}

The Buchberger Criterion enables us to formulate the following procedure for computing Gr\"obner bases of two-sided ideals. Note that, in the procedure, by a \emph{fair strategy} we mean a selection strategy which ensures that every obstruction is selected eventually. Since these Gr\"obner bases need not be finite, we have to content ourselves with an enumerating procedure.

\begin{theorem}[(The Buchberger Procedure)]\label{BuProc}
Let $s\geq 1$, and let $G=\{g_1,\dots,g_s\}\subseteq K\langle X\rangle$ be a set of non-zero polynomials which generate a two-sided ideal $I=\langle G\rangle\subseteq K\langle X\rangle$. Consider the following sequence of instructions.
\begin{itemize}
\item[{\rm (B1)}] Let $B=\bigcup_{1\leq i\leq j\leq s}{\rm NTObs}(i,j)$.

\item[{\rm (B2)}] If $B=\emptyset$, return the result ${G}$. Otherwise, select an obstruction ${\rm o}_{i,j}(w_{i},w'_{i};$ $w_{j},w'_{j})\in B$ using a fair strategy and delete it from $B$.

\item[{\rm (B3)}] Compute the S-polynomial $S=S_{i,j}(w_{i},w'_{i};w_{j},w'_{j})$ and its normal remainder $S'={\rm NR}_{\sigma,G}(S)$. If $S'=0$, continue with step {\rm (B2)}.

\item[{\rm (B4)}] Increase $s$ by one, append $g_{s}=S'$ to the set ${G}$, and append the set of obstructions $\bigcup_{1\leq i\leq s}{\rm NTObs}(i,s)$ to the set $B$. Then continue with step~{\rm (B2)}.
\end{itemize}
This is a procedure that enumerates a ${\sigma}$-Gr\"obner basis ${G}$ of $I$. If $I$ has a finite ${\sigma}$-Gr\"obner basis, the procedure stops after finitely many steps and the resulting set ${G}$ is a finite ${\sigma}$-Gr\"obner basis of $I$.
\end{theorem}
\begin{proof}
Note that this is a straightforward generalization of the commutative version of Buchberger's algorithm to the non-commutative case. We refer to~\cite{Mo94} for the original form of this procedure and to \cite{Xiu12}, Theorem 4.1.14 for a detailed proof.
\end{proof}

%\section{Non-Commutative Gebauer-M\"oller Criteria}
\section{Non-Commutative Gebauer-M\"oller Criteria}
\label{sec3}

In this section we present non-commutative Gebauer-M\"oller criteria. They check whether an obstruction can be represented by ``smaller'' obstructions. If so, we declare such obstructions to be \emph{unnecessary}.

In the following, let $s\geq1$, and let $G=\{g_1,\dots,g_{s}\}\subseteq K\langle X\rangle\setminus\{0\}$ be a set of non-commutative polynomials. Recall that $F_s$ is the free two-sided $K\langle X\rangle$-module of rank $s$ and $\mathbb{T}(F_{s})=\{we_{i}w'\ |\ i\in\{1,\dots,s\}, w,w'\in\langle X\rangle\}$ is the set of terms in $F_{s}$. Before going into details, we define a certain well-ordering $\tau$ on $\mathbb{T}(F_{s})$ and use it to order obstructions.

\begin{definition}\label{DefMTOrd}
Let us define a relation $\tau$ on $\mathbb{T}(F_{s})$ as follows.
For two terms $w_{1}e_{i}w'_{1},w_{2}e_{j}w'_{2}\in\mathbb{T}(F_{s})$, we let $w_{1}e_{i}w'_{1}\geq_{\tau}w_{2}e_{j}w'_{2}$ if
\begin{itemize}
\item[(a)] $w_{1}\lw_{\sigma}(g_{i})w'_{1}>_{\sigma}w_{2}\lw_{\sigma}(g_{j})w'_{2}$, or

\item[(b)] $w_{1}\lw_{\sigma}(g_{i})w'_{1}=w_{2}\lw_{\sigma}(g_{j})w'_{2}$ and $i>j$, or

\item[(c)] $w_{1}\lw_{\sigma}(g_{i})w'_{1}=w_{2}\lw_{\sigma}(g_{j})w'_{2}$ and $i=j$ and $w_{1}\geq_{\sigma}w_{2}$.
\end{itemize}
One can check that $\tau$ is a well-ordering and is compatible with scalar multiplication. The relation $\tau$ is called the \emph{module term ordering induced by} $({\sigma},{G})$ on $\mathbb{T}(F_{s})$.
\end{definition}

By definition, for every obstruction ${\rm o}_{i,j}(w_{i},w'_{i};w_{j},w'_{j})\in\bigcup_{1\leq i\leq j\leq s}{\rm Obs}(i,j)$, we have $w_{i}e_{i}w'_{i}<_{\tau}w_{j}e_{j}w'_{j}$. We extend the ordering $\tau$ to the set of obstructions $\bigcup_{1\leq i\leq j\leq s}{\rm Obs}(i,j)$ by committing the following slight abuse of notation.

\begin{definition}\label{DefObsOrd}
Let $\tau$ be the module term ordering induced by $({\sigma},{G})$ on $\mathbb{T}(F_{s})$. Let ${\rm o}_{i,j}(w_{i},w'_{i};w_{j},w'_{j}), {\rm o}_{k,l}(w_{k},w'_{k};w_{l},w'_{l})$ be two obstructions in the set $\bigcup_{1\leq i\leq j\leq s}{\rm Obs}(i,j)$. If we have $w_{j}e_{j}w'_{j}>_{\tau}w_{l}e_{l}w'_{l}$, or if we have $w_{j}e_{j}w'_{j}=w_{l}e_{l}w'_{l}$ and $w_{i}e_{i}w'_{i}\geq_{\tau}w_{k}e_{k}w'_{k}$, then we let ${\rm o}_{i,j}(w_{i},w'_{i};w_{j},w'_{j})\geq_{\tau}{\rm o}_{k,l}(w_{k},w'_{k};$ $w_{l},w'_{l})$. The ordering $\tau$ is called the ordering \emph{induced by} $({\sigma},{G})$ on the set of obstructions $\bigcup_{1\leq i\leq j\leq s}{\rm Obs}(i,j)$.
\end{definition}

One can verify that $\tau$ is also a well-ordering on $\bigcup_{1\leq i\leq j\leq s}{\rm Obs}(i,j)$ and compatible with scalar multiplication. Recall that ${\rm o}_{j,i}(w_{j},w'_{j};w_{i},w'_{i})=-{\rm o}_{i,j}(w_{i},w'_{i};$ $w_{j},w'_{j})$, we can generalize the ordering $\tau$ to the set $\bigcup_{1\leq i,j\leq s}{\rm Obs}(i,j)$ by letting ${\rm o}_{j,i}(w_{j},w'_{j};w_{i},w'_{i})=_{\tau}{\rm o}_{i,j}(w_{i},w'_{i};w_{j},w'_{j})$.

Now we are ready to generalize the commutative Gebauer-M\"oller criteria (see \cite{CKR04} and \cite{GM88}) to the non-commutative case. Recall that, in step (B4) of the Buchberger Procedure, when a new generator $g_{s}$ is added, we immediately construct new obstructions $\bigcup_{1\leq i\leq s}{\rm NTObs}(i,s)$. We want to detect unnecessary obstructions in the set $\bigcup_{1\leq i\leq s}{\rm NTObs}(i,s)$ of newly constructed obstructions as well as in the set $\bigcup_{1\leq i\leq j\leq s-1}{\rm NTObs}(i,j)$ of previously constructed obstructions. We achieve this goal via the following three steps. Firstly, we detect unnecessary obstructions in the set $\bigcup_{1\leq i\leq s}{\rm NTObs}(i,s)$ with the aid of other obstructions also in this set. This step is called a \emph{head reduction step} in \cite{CKR04}. Secondly, we detect unnecessary obstructions in the set $\bigcup_{1\leq i\leq s}{\rm NTObs}(i,s)$ with the aid of obstructions in the set $\bigcup_{1\leq i\leq j\leq s-1}{\rm NTObs}(i,j)$. This step is called a \emph{tail reduction step} in \cite{CKR04}. Thirdly, we detect unnecessary obstructions in the set $\bigcup_{1\leq i\leq j\leq s-1}{\rm NTObs}(i,j)$ with the aid of the new generator $g_{s}$. Indeed, the first step corresponds to the commutative Gebauer-M\"oller criteria $M$ and $F$, and the last step corresponds to criterion $B_k$ (c.f. \cite{GM88}, Subsection 3.4).

The following definition proves quite helpful for presenting our idea.

\begin{definition}
A \emph{related pair of $s$-obstructions} is a pair of two distinct non-trivial obstructions ${\rm o}_{i,s}(w_{i},w'_{i};u_{s},u'_{s}), {\rm o}_{j,s}(w_{j},w'_{j};v_{s},v'_{s})\in\bigcup_{1\leq i\leq s}{\rm NTObs}(i,s)$ such that there exist two words
$w,w'\in\langle X\rangle$ satisfying $u_{s}=$ $wv_{s}$ and $u'_{s}=v'_{s}w'$.
\end{definition}

The following lemma is the key to implement the first step, that is, to detect unnecessary obstructions in the set $\bigcup_{1\leq i\leq s}{\rm NTObs}(i,s)$ of newly constructed obstructions via other obstructions in this set.

\begin{lemma}\label{LemMLW}
Let ${\rm o}_{i,s}(w_{i},w'_{i};u_{s},u'_{s})$ and ${\rm o}_{j,s}(w_{j},w'_{j};v_{s},v'_{s})$ be a related pair of $s$-obstructions with two words $w,w'\in\langle X\rangle$ satisfying $u_{s}=wv_{s}$ and $u'_{s}=v'_{s}w'$. Then we have
$$
{\rm o}_{i,s}(w_{i},w'_{i};u_{s},u'_{s})=w{\rm o}_{j,s}(w_{j},w'_{j};v_{s},v'_{s})w'+{\rm o}_{i,j}(w_{i},w'_{i};ww_{j},w'_{j}w').
$$
If the S-polynomials $S_{j,s}(w_{j},w'_{j};v_{s},v'_{s})$ and $S_{i,j}(w_{i},w'_{i};ww_{j},w'_{j}w')$ have Gr\"obner representations in terms of $G$, then so does $S_{i,s}(w_{i},w'_{i};u_{s},u'_{s})$. Moreover, the inequalities ${\rm o}_{i,s}(w_{i},w'_{i};u_{s},u'_{s})>_{\tau}{\rm o}_{j,s}(w_{j},w'_{j};v_{s},v'_{s})$ and ${\rm o}_{i,s}(w_{i},w'_{i};u_{s},u'_{s})>_{\tau}{\rm o}_{i,j}(w_{i},w'_{i};ww_{j},w'_{j}w')$ hold if one of the following conditions holds: (a) $ww'\neq 1$; (b) $i>j$; (c) $i=j$ and $ww'=1$ and $w_i>_{\sigma}w_j$.
\end{lemma}

\begin{proof}
The equality follows from Definition \ref{DefObs} and the conditions $u_{s}=wv_{s}$ and $u'_{s}=v'_{s}w'$. We show that, if $S_{j,s}(w_{j},w'_{j};v_{s},v'_{s})$ and $S_{i,j}(w_{i},w'_{i};ww_{j},w'_{j}w')$
have Gr\"obner representations in terms of $G$, then so does $S_{i,s}(w_{i},w'_{i};u_{s},u'_{s})$.
Clearly we have
$$
S_{i,s}(w_{i},w'_{i};u_{s},u'_{s})=wS_{j,s}(w_{j},w'_{j};v_{s},v'_{s})w'+S_{i,j}(w_{i},w'_{i};ww_{j},w'_{j}w').
$$
Without loss of generality, we assume that $S_{i,s}(w_{i},w'_{i};u_{s},u'_{s}), S_{j,s}(w_{j},w'_{j};v_{s},v'_{s})$
and $S_{i,j}(w_{i},w'_{i};ww_{j},w'_{j}w')$ are non-zero. Since there is a Gr\"obner representation for
$S_{j,s}(w_{j},w'_{j};v_{s},v'_{s})$, we have
$$
S_{j,s}(w_{j},w'_{j};v_{s},v'_{s})=\sum^\mu_{k=1}a_kw_kg_{i_k}w'_k
$$
with $a_k\in K\setminus\{0\},\ w_k,w'_k\in\langle X\rangle,\ g_{i_k}\in G$ for all
$k\in\{1,\dots,\mu\}$, such that
$\lw_{\sigma}(v_{s}g_{s}v'_{s})>_{\sigma}\lw_{\sigma}(a_kw_kg_{i_k}w'_k)$. Similarly, for $S_{i,j}(w_{i},w'_{i};ww_{j},w'_{j}w')$ we have
$$
S_{i,j}(w_{i},w'_{i};ww_{j},w'_{j}w')=\sum^\nu_{l=1}b_lw_lg_{i_l}w'_l
$$
with $b_l\in K\setminus\{0\},\ w_l,w'_l\in\langle X\rangle,\ g_{i_l}\in G$ for all
$l\in\{1,\dots,\nu\}$, such that $\lw_{\sigma}(ww_{j}g_{j}w'_{j}w')>_{\sigma}\lw_{\sigma}(b_lw_lg_{i_l}w'_l)$.
Therefore we have
\begin{eqnarray*}
S_{i,s}(w_{i},w'_{i};u_{s},u'_{s})&=&w(\sum^\mu_{k=1}a_kw_kg_{i_k}w'_k)w'+\sum^\nu_{l=1}b_lw_lg_{i_l}w'_l\\
&=&\sum^\mu_{k=1}a_kww_kg_{i_k}w'_kw'+\sum^\nu_{l=1}b_lw_lg_{i_l}w'_l.
\end{eqnarray*}
As $u_{s}\lw_{\sigma}(g_{s})u'_{s}=wv_{s}\lw_{\sigma}(g_{s})v'_{s}w'$, we have
$\lw_{\sigma}(u_{s}g_{s}u'_{s})=\lw_{\sigma}(wv_{s}g_{s}v'_{s}w')$ $>_{\sigma}\lw_{\sigma}(ww_kg_{i_k}w'_kw')$
for all $k\in\{1,\dots,\mu\}$. By Definition~\ref{DefObs}, we have
$\lw_{\sigma}(u_{s}g_{s}u'_{s})=\lw_{\sigma}(w_{i}g_iw'_{i})=\lw_{\sigma}(ww_{j}g_jw'_{j}w')>_{\sigma}\lw_{\sigma}(b_lw_lg_{i_l}w'_l)$ for all $l\in\{1,\dots,\nu\}$. Therefore
$$
S_{i,s}(w_{i},w'_{i};u_{s},u'_{s})=\sum^\mu_{k=1}a_kww_kg_{i_k}w'_kw'+\sum^\nu_{l=1}b_lw_lg_{i_l}w'_l
$$
is a Gr\"obner representation of $S_{i,s}(w_{i},w'_{i};u_{s},u'_{s})$.

Finally, the inequalities follow from the conditions $u_{s}=wv_{s}$ and $u'_{s}=v'_{s}w'$, and from Definitions \ref{DefMTOrd} and \ref{DefObsOrd}.
\end{proof}

The following example shows that the obstruction ${\rm o}_{i,j}(w_{i},w'_{i};ww_{j},w'_{j}w')$
in Lemma~\ref{LemMLW} can be trivial. Similar phenomena also occur in Lemmas~\ref{LemTR} and~\ref{LemBk} below.

\begin{example}\label{ExMC}
Consider polynomials $G=\{g_{1},g_{2},g_{3}\}$ in the non-commutative polynomial ring $K\langle x,y\rangle$.
\begin{itemize}
\item[(a)] Assume that $\lm_{\sigma}(g_{1})=(xy)^{2}, \lm_{\sigma}(g_{2})=y$ and
$\lm_{\sigma}(g_{3})=xyx^{2}y$. Then we have ${\rm o}_{1,3}(xyx,1;1,xy), {\rm o}_{2,3}(x,x^{2}y;1,1)\in\bigcup_{1\leq i\leq 3}{\rm NTObs}(i,3)$, and
$${\rm o}_{1,3}(xyx,1;1,xy)={\rm o}_{2,3}(x,x^{2}y;1,1)xy+{\rm o}_{1,2}(xyx,1;x,x^{2}yxy).$$
One can check that ${\rm o}_{1,2}(xyx,1;x,x^{2}yxy)$ is an obstruction without overlap.

\item[(b)] Now assume that $\lm_{\sigma}(g_{1})=y^3, \lm_{\sigma}(g_{2})=x^2y^{2}$ and
$\lm_{\sigma}(g_{3})=xyx^{2}y$. Then we have ${\rm o}_{1,3}(xyx^{2},1;1,y^{2}), {\rm o}_{2,3}(xy,1;1,y)\in\bigcup_{1\leq i\leq 3}{\rm NTObs}(i,3)$, and
$${\rm o}_{1,3}(xyx^{2},1;1,y^{2})={\rm o}_{2,3}(xy,1;1,y)y+{\rm o}_{1,2}(xyx^2,1;xy,y).$$
Observe that ${\rm o}_{1,2}(xyx^2,xy;y)=xy{\rm o}_{1,2}(x^2,1;1,y)$ is a proper multiple of
the non-trivial obstruction ${\rm o}_{1,2}(x^2,1;1,y)$.
\end{itemize}
\end{example}

In the following, we present the non-commutative multiply criterion and the leading word
criterion. They are non-commutative analogues of the Gebauer-M\"oller criteria M and F, respectively.

\begin{proposition2}[(Non-Commutative Multiply Criterion)]\label{ncMCrit}
Suppose that ${\rm o}_{i,s}(w_{i},w'_{i};u_{s},u'_{s})$ and ${\rm o}_{j,s}(w_{j},w'_{j};v_{s},v'_{s})$ are a related pair of $s$-obstructions with two words $w,w'\in\langle X\rangle$ satisfying $u_{s}=wv_{s}$ and $u'_{s}=v'_{s}w'$. Then we can remove the obstruction ${\rm o}_{i,s}(w_{i},w'_{i};u_{s},u'_{s})$ from $\bigcup_{1\leq i\leq s}{\rm NTObs}(i,s)$
in the execution of the Buchberger Procedure if $ww'\neq 1$.
\end{proposition2}

\begin{proof}
By the previous lemma, the obstruction ${\rm o}_{i,s}(w_{i},w'_{i};u_{s},u'_{s})$
can be represented as
$$
{\rm o}_{i,s}(w_{i},w'_{i};u_{s},u'_{s})=w{\rm o}_{j,s}(w_{j},w'_{j};v_{s},v'_{s})w'+{\rm o}_{i,j}(w_i,w'_{i};ww_{j},w'_{j}w').
$$
Observe that the condition $ww'\neq 1$ corresponds to condition (a) of Lemma~\ref{LemMLW}. By Lemma~\ref{LemMLW} and Proposition~\ref{BuCrit}, it suffices to show that $S_{j,s}(w_{j},w'_{j};v_{s},v'_{s})$ and $S_{i,j}(w_i,w'_{i};ww_{j},w'_{j}w')$ have
Gr\"obner representations in terms of $G$. Theorem~\ref{BuProc} ensures that
$S_{j,s}(w_{j},w'_{j};v_{s},v'_{s})$ has a Gr\"obner representation in terms of $G$.
Note that the obstruction ${\rm o}_{i,j}(w_i,w'_{i};ww_{j},w'_{j}w')$ can be an obstruction without overlap or a (proper) multiple of a non-trivial obstruction (for instance, see Example~\ref{ExMC}). If ${\rm o}_{i,j}(w_i,w'_{i};ww_{j},w'_{j}w')$ is an obstruction without overlap, then by Lemma~\ref{ProdCrit} its S-polynomial has a Gr\"obner representation in terms of~$G$. If ${\rm o}_{i,j}(w_i,w'_{i};ww_{j},w'_{j}w')$ is a multiple of a non-trivial obstruction,
then Lemma~\ref{sec2lem7} and Theorem~\ref{BuProc} guarantee that
$S_{i,j}(w_i,w'_{i};ww_{j},w'_{j}w')$ has a Gr\"obner representation in terms of $G$.
\end{proof}

\begin{proposition2}[(Non-Commutative Leading Word Criterion)]\label{ncLWCrit}
Suppose that ${\rm o}_{i,s}(w_{i},w'_{i};u_{s},u'_{s})$ and ${\rm o}_{j,s}(w_{j},w'_{j};v_{s},v'_{s})$ are a related pair of $s$-obstructions with two words $w,w'\in\langle X\rangle$ satisfying $u_{s}=wv_{s}$ and $u'_{s}=v'_{s}w'$. Then ${\rm o}_{i,s}(w_{i},w'_{i};u_{s},u'_{s})$ can be removed from
$\bigcup_{1\leq i\leq s}{\rm NTObs}(i,s)$ in the execution of the Buchberger Procedure
if one of the following conditions is satisfied: (a) $i>j$; (b) $i=j$ and $ww'=1$ and $w_{i}>_{\sigma}w_{j}$.
\end{proposition2}

\begin{proof}
Observe that the conditions (a) and (b) correspond to conditions (b) and (c) of Lemma~\ref{LemMLW}, respectively. The claim follows in the same way as in the proof of Proposition \ref{ncMCrit}.
\end{proof}

Next we work on detecting unnecessary obstructions in $\bigcup_{1\leq i\leq s}{\rm NTObs}(i,s)$
via obstructions in the set $\bigcup_{1\leq i\leq j\leq s-1}{\rm NTObs}(i,j)$ of previously constructed obstructions.

\begin{lemma}\label{LemTR}
Let ${\rm o}_{j,s}(u_{j},u'_{j};w_{s},w'_{s})$ and ${\rm o}_{i,j}(w_{i},w'_{i};v_{j},v'_{j})$ be non-trivial obstructions in
$\bigcup_{1\leq i\leq s}{\rm NTObs}(i,s)$ and $\bigcup_{1\leq i\leq j\leq s-1}{\rm NTObs}(i,j)$, respectively. If there exist two words
$w,w'\in\langle X\rangle$ such that $u_{j}=wv_{j}$ and $u'_{j}=v'_{j}w'$,
then we have
$$
{\rm o}_{j,s}(u_{j},u'_{j};w_{s},w'_{s})=-w{\rm o}_{i,j}(w_{i},w'_{i};v_{j},v'_{j})w'+
{\rm o}_{i,s}(ww_{i},w'_{i}w';w_{s},w'_{s})
$$
where the inequalities ${\rm o}_{j,s}(u_{j},u'_{j};w_{s},w'_{s})
>_{\tau}{\rm o}_{i,j}(w_{i},w'_{i};v_{j},v'_{j})$
and ${\rm o}_{j,s}(u_{j},u'_{j};$ $w_{s},w'_{s})
>_{\tau}{\rm o}_{i,s}(ww_{i},w'_{i}w';w_{s},w'_{s})$ hold.
Further, if the S-polynomials $S_{i,j}(w_{i},w'_{i};$ $v_{j},v'_{j})$ and
$S_{i,s}(ww_{i},w'_{i}w';w_{s},w'_{s})$ have Gr\"obner representations in terms of $G$,
then so does $S_{j,s}(u_{j},u'_{j};w_{s},w'_{s})$.
\end{lemma}

\begin{proof}
The claimed equality follows from Definition~\ref{DefObs} and from the conditions $u_{j}=wv_{j}$ and
$u'_{j}=v'_{j}w'$. We have ${\rm o}_{j,s}(u_{j},u'_{j}; w_{s},w'_{s})>_{\tau}{\rm o}_{i,j}(w_{i},w'_{i};v_{j},v'_{j})$ for $w_{s}e_{s}w'_{s}>_{\tau}u_{j}e_{j}u'_{j}=
wv_{j}e_{j}v'_{j}w\geq_{\tau}v_{j}e_{j}v'_{j}$. From the inequality $u_{j}e_{j}u'_{j}=
wv_{j}e_{j}v'_{j}w>_{\tau}ww_{i}e_{i}w'_{i}w'$, it follows that ${\rm o}_{j,s}(u_{j},u'_{j};w_{s},w'_{s})
>_{\tau}{\rm o}_{i,s}(ww_{i},w'_{i}w';$ $w_{s},w'_{s})$. Again, we can prove the second part
by following the same argument as in the proof of the second part of Lemma~\ref{LemMLW}.
\end{proof}

Note that the obstruction ${\rm o}_{i,s}(ww_{i},w'_{i}w';w_{s},w'_{s})$ in Lemma~\ref{LemTR}
can be an obstruction without overlap or a multiple of a non-trivial obstruction.
However, it suffices for us to consider only the former case, since the latter case has been
considered in Proposition~\ref{ncLWCrit} under its condition (a).

\begin{proposition2}[(Non-Commutative Tail Reduction)]\label{ncTR}
Suppose that ${\rm o}_{j,s}(u_{j},u'_{j};w_{s},w'_{s})$ and
${\rm o}_{i,j}(w_{i},w'_{i};v_{j},v'_{j})$
are two non-trivial obstructions as in Lemma \ref{LemTR}.
If the word $ww_{i}$ is a multiple of $w_{s}\lw_{\sigma}(g_{s})$, or if the word $w'_{i}w'$
is a multiple of $\lw_{\sigma}(g_{s})w'_{s}$, then ${\rm o}_{j,s}(u_{j},u'_{j};w_{s},w'_{s})$
can be removed from $\bigcup_{1\leq i\leq s}{\rm NTObs}(i,s)$ in the execution of the Buchberger Procedure.
\end{proposition2}

\begin{proof}
By Lemma~\ref{LemTR}, the obstruction ${\rm o}_{j,s}(u_{j},u'_{j};w_{s},w'_{s})$ can be represented as
$$
{\rm o}_{j,s}(u_{j},u'_{j};w_{s},w'_{s})=
-w{\rm o}_{i,j}(w_{i},w'_{i};v_{j},v'_{j})w'+{\rm o}_{i,s}(ww_{i},w'_{i}w';w_{s},w'_{s})
$$
By Lemma~\ref{LemTR} and Proposition~\ref{BuCrit}, it suffices to show that $S_{i,j}(w_{i},w'_{i};v_{j},v'_{j})$ and $S_{i,s}(ww_{i},w'_{i}w';w_{s},w'_{s})$ have
Gr\"obner representations in terms of $G$. Theorem~\ref{BuProc} ensures that
$S_{i,j}(w_{i},w'_{i};v_{j},v'_{j})$ has a Gr\"obner representation
in terms of $G$. Note that $ww_{i}$ is a multiple of $w_{s}\lw_{\sigma}(g_{s})$ or $w'_{i}w'$ is a multiple
of $\lw_{\sigma}(g_{s})w'_{s}$. This implies that ${\rm o}_{i,s}(ww_{i},$ $w'_{i}w';w_{s},w'_{s})$ has
no overlap. Then, by Lemma~\ref{ProdCrit}, $S_{i,s}(ww_{i},w'_{i}w';w_{s},w'_{s})$ has a Gr\"obner representation
in terms of $G$.
\end{proof}

\begin{example}\label{ExTR}
Consider polynomials $G=\{g_1,g_2,g_3\}$ in the non-commutative polynomial ring $K\langle x,y\rangle$. Assume that $\lm_{\sigma}(g_{1})=xy, \lm_{\sigma}(g_{2})=(xy)^2$ and $\lm_{\sigma}(g_{3})=xyx$. Then we have the following non-trivial obstructions.
\begin{eqnarray*}
{\rm NTOb}(1,1)&=&\emptyset\\
{\rm NTOb}(1,2)\cup{\rm NTOb}(2,2)&=&\{{\rm o}_{1,2}(xy,1;1,1),\ {\rm o}_{1,2}(1,xy;1,1),\\
&&\ {\rm o}_{2,2}(1,xy;xy,1)\}\\
{\rm NTOb}(1,3)\cup{\rm NTOb}(2,3)\cup{\rm NTOb}(3,3)&=&\{{\rm o}_{1,3}(1,x;1,1),\ {\rm o}_{1,3}(xy,1;1,y),\\
&&\ {\rm o}_{2,3}(1,1;1,y),\ {\rm o}_{2,3}(xy,1;1,yxy),\\
&&\ {\rm o}_{3,3}(1,yx;xy,1)\}
\end{eqnarray*}
The obstruction ${\rm o}_{2,3}(xy,1;1,yxy)$ can be detected by the Non-Commutative Tail Reduction, since
$${\rm o}_{2,3}(xy,1;1,yxy)=-xy{\rm o}_{1,2}(xy,1;1,1)+{\rm o}_{1,3}(xyxy,1;1,xyx)$$
and the obstruction ${\rm o}_{1,3}(xyxy,1;1,xyx)$ has no overlap.
\end{example}

\begin{remark}\label{RemTR}
Our experiments show that, after applying the previous two criteria, the Non-Commutative
Tail Reduction is unlikely to happen in the Buchberger Procedure. This may be due to the fact that, frequently, the obstruction ${\rm o}_{i,j}(w_{i},w'_{i};v_{j},v'_{j})$ in the equation of Lemma~\ref{LemTR} or Proposition~\ref{ncTR} had been removed by the Non-Commutative Multiply Criterion and the Non-Commutative Leading Word Criterion before the Non-Commutative Tail Reduction is applied. Consider Example \ref{ExTR} again. It is easy to check that the obstruction ${\rm o}_{1,2}(xy,1;1,1)$ can be detected as an unnecessary obstruction by the Non-Commutative Leading Word Criterion with the help of ${\rm o}_{1,2}(1,xy;1,1)$.
\end{remark}

So far we have detected unnecessary obstructions in the set $\bigcup_{1\leq i\leq s}{\rm O}(i,s)$ of newly constructed obstructions. Intuitively, we are also able to detect unnecessary obstructions in the set $\bigcup_{1\leq i\leq j\leq s-1}{\rm Obs}(i,j)$ of previously constructed obstructions. Thus, in the last step, we detect unnecessary obstructions in this set by using the new generator $g_{s}$.

\begin{lemma}\label{LemBk}
Let  ${\rm o}_{i,j}(w_{i},w'_{i};w_{j},w'_{j})\in\bigcup_{1\leq i\leq j\leq s-1}{\rm NTObs}(i,j)$ be a non-trivial obstruction. If there are two words $w,w'\in\langle X\rangle$ satisfying $w_{j}\lw_{\sigma}(g_{j})w'_{j}=w\lw_{\sigma}(g_{s})w'$, then we can represent ${\rm o}_{i,j}(w_{i},w'_{i};w_{j},w'_{j})$ as
$${\rm o}_{i,j}(w_{i},w'_{i};w_{j},w'_{j})={\rm o}_{i,s}(w_{i},w'_{i};w,w')-{\rm o}_{j,s}(w_{j},w'_{j};w,w').$$
Moreover, if $S_{i,s}(w_{i},w'_{i};w,w')$ and $S_{j,s}(w_{j},w'_{j};w,w')$ have Gr\"obner representations in terms of $G$, then so does $S_{i,j}(w_{i},w'_{i};w_{j},w'_{j})$.
\end{lemma}

\begin{proof}
The equality  follows from Definition~\ref{DefObs} and from the
condition that $w_{j}\lw_{\sigma}(g_{j})w'_{j}=w\lw_{\sigma}(g_{s})w'$.
The proof of the second part is analogous to the proof of the second part of Lemma~\ref{LemMLW}.
\end{proof}

The following example shows that the obstruction ${\rm o}_{i,s}(w_{i},w'_{i};w,w')$ in the equation of Lemma~\ref{LemBk} can be an obstruction without overlap or a (proper) multiple of a non-trivial obstruction. In the case that ${\rm o}_{i,s}(w_{i},w'_{i};w,w')$ is a multiple of a non-trivial obstruction, say ${\rm o}_{i,s}(\tilde{w}_{i},\tilde{w}'_{i};\tilde{w},\tilde{w}')$, the example shows that it is not necessary to have ${\rm o}_{i,s}(w_{i},w'_{i};w,w')>_{\tau}{\rm o}_{i,s}(\tilde{w}_{i},\tilde{w}'_{i};\tilde{w},\tilde{w}')$ (compared to Lemmas~\ref{LemMLW} and~\ref{LemTR}). The same also holds for the obstruction ${\rm o}_{j,s}(w_{j},w'_{j};w,w')$ in the equation of Lemma~\ref{LemBk}.

\begin{example}
Consider polynomials $G=\{g_{1},g_{2},g_{3}\}$ in the non-commutative polynomial ring $K\langle x,y\rangle$ with $\lm_{\sigma}(g_{1})=x^3yx, \lm_{\sigma}(g_{2})=x^2$ and $\lm_{\sigma}(g_{3})=x$. We have ${\rm o}_{1,2}(1,1;x,yx)\in\bigcup_{1\leq i\leq j\leq 2}{\rm NTObs}(i,j)$ and $x\lw_{\sigma}(g_{2})yx=x^3yx=x^3y\lw_{\sigma}(g_{3})$ and
$${\rm o}_{1,2}(1,1;x,yx)={\rm o}_{1,3}(1,1;x^3y,1)-{\rm o}_{2,3}(x,yx;x^3y,1).$$
One can check that ${\rm o}_{1,3}(1,1;x^3y,1)$ is a non-trivial obstruction in ${\rm NTObs}(1,3)$ and ${\rm o}_{1,2}(1,1;x,yx)<_{\tau}{\rm o}_{1,3}(1,1;x^3y,1)$. Moreover, ${\rm o}_{2,3}(x,yx;x^3y,1)$ is an obstruction without overlap.
\end{example}

The following is a non-commutative analogue of the Gebauer-M\"oller criterian $B_k$, which is also known as the \emph{chain criterion} (or \emph{criterion 1}) of Buchberger (cf. \cite{Bu85}).

\begin{proposition2}[(Non-Commutative Backward Criterion)]\label{ncBKCrit}
Suppose that ${\rm o}_{i,j}(w_{i},w'_{i};w_{j},w'_{j})\in\bigcup_{1\leq i\leq j\leq s-1}{\rm NTObs}(i,j)$ is a non-trivial obstruction. Then in the execution of the Buchberger Procedure ${\rm o}_{i,j}(w_{i},w'_{i};w_{j},w'_{j})$ can be removed from $\bigcup_{1\leq i\leq j\leq s-1}{\rm NTObs}(i,j)$ if the following three conditions are satisfied.
\begin{itemize}
\item[{\rm (a)}] There are two words $w,w'\in\langle X\rangle$ such that $w_{j}\lw_{\sigma}(g_{j})w'_{j}=w\lw_{\sigma}(g_{s})w'$.

\item[{\rm (b)}] The obstruction ${\rm o}_{i,s}(w_{i},w'_{i};w,w')$ is either an obstruction without overlap or a (proper) multiple of a non-trivial obstruction in $\bigcup_{1\leq i\leq s}{\rm NTObs}(i,s)$.

\item[{\rm (c)}] The obstruction ${\rm o}_{j,s}(w_{j},w'_{j};w,w')$ is either an obstruction without overlap or a (proper) multiple of a non-trivial obstruction in $\bigcup_{1\leq i\leq s}{\rm NTObs}(i,s)$.
\end{itemize}
\end{proposition2}

\begin{proof}
By Lemma~\ref{LemBk}, we can represent ${\rm o}_{i,j}(w_{i},w'_{i};w_{j},w'_{j})$ as
$${\rm o}_{i,j}(w_{i},w'_{i};w_{j},w'_{j})={\rm o}_{i,s}(w_{i},w'_{i};w,w')-{\rm o}_{j,s}(w_{j},w'_{j};w,w').$$
By Lemma~\ref{LemBk} and Proposition~\ref{BuCrit}, it suffices to show that the S-polynomials $S_{i,s}(w_{i},w'_{i};w,w')$
and $S_{j,s}(w_{j},w'_{j};w,w')$ have Gr\"obner representations in terms of $G$.
If ${\rm o}_{i,s}(w_{i},w'_{i};w,w')$ is an obstruction without overlap, then,
by Lemma~\ref{ProdCrit}, its S-polynomial has a Gr\"obner representations in terms of $G$.
If it is a multiple of a non-trivial obstruction in $\bigcup_{1\leq i\leq s}{\rm NTObs}(i,s)$,
then Lemma~\ref{ProdCrit} and Theorem~\ref{BuProc} ensure that $S_{i,s}(w_{i},w'_{i};w,w')$
has a Gr\"obner representations in terms of $G$. By the same argument, one can show that
$S_{j,s}(w_{j},w'_{j};w,w')$ has a Gr\"obner representations in terms of $G$.
\end{proof}

We would like to mention that the Non-Commutative Backward Criterion given in Proposition~\ref{ncBKCrit} covers in particular all useless obstructions presented by T. Mora in \cite{Mo94}, Lemma 5.11.

\begin{remark}
In order to apply Propositions~\ref{ncMCrit},~\ref{ncLWCrit},~\ref{ncTR} and~\ref{ncBKCrit} to remove unnecessary obstructions during the execution of the Buchberger Procedure, it is crucial to make sure that the S-polynomials of those removed obstructions have Gr\"obner representations.

\begin{itemize}
\item[(a)] Propositions~\ref{ncMCrit},~\ref{ncLWCrit} and~\ref{ncTR} remove unnecessary
non-trivial obstructions, say ${\rm o}_{i,s}(w_{i},w'_{i};w_{s},w'_{s})$, from the set
$\bigcup_{1\leq i\leq s}{\rm NTObs}(i,s)$ of newly constructed obstructions.
The Gr\"obner representation of the S-polynomial $S_{i,s}(w_{i},w'_{i};w_{s},w'_{s})$
depends on the Gr\"obner representations of the S-poly\-no\-mials of two smaller obstructions
in the set $\bigcup_{1\leq i\leq j\leq s-1}{\rm Obs}(i,j)$ and
the set $\bigcup_{1\leq i\leq s}{\rm Obs}(i,s)$.

\item[(b)] Proposition~\ref{ncBKCrit} removes unnecessary obstructions, say ${\rm o}_{i,j}(w_{i},w'_{i};w_{j},w'_{j})$, from the set $\bigcup_{1\leq i\leq j\leq s-1}{\rm Obs}(i,j)$ of previously constructed obstructions. The Gr\"obner representation of $S_{i,j}(w_{i},w'_{i};w_{j},w'_{j})$ depends on the Gr\"obner representations of the S-polynomials of two obstructions, say ${\rm o}_{k,s}(w_{k},w'_{k};$ $u_{s},u'_{s})$ and ${\rm o}_{l,s}(w_{l},w'_{l};v_{s},v'_{s})$, in $\bigcup_{1\leq i\leq s}{\rm Obs}(i,s)$, which are not necessarily smaller than ${\rm o}_{i,j}(w_{i},w'_{i};w_{j},w'_{j})$. If ${\rm o}_{k,s}(w_{k},w'_{k};u_{s},u'_{s})$ is a multiple of a non-trivial obstruction, say ${\rm o}_{k,s}(\tilde{w}_{k},\tilde{w}'_{k};\tilde{u}_{s},\tilde{u}'_{s})$, in $\bigcup_{1\leq i\leq s}{\rm NTObs}(i,s)$, then, before removing ${\rm o}_{i,j}(w_{i},w'_{i};w_{j},w'_{j})$, it is important to ensure that ${\rm o}_{k,s}(\tilde{w}_{k},\tilde{w}'_{k};\tilde{u}_{s},\tilde{u}'_{s})$ is in $\bigcup_{1\leq i\leq s}{\rm NTObs}(i,s)$. The same check should be applied to ${\rm o}_{l,s}(w_{l},w'_{l};v_{s},v'_{s})$.
\end{itemize}
Observe that Propositions~\ref{ncMCrit},~\ref{ncLWCrit} and~\ref{ncBKCrit} are actually generalizations of the well-known Gebauer-M\"oller criteria (see \cite{CKR04} and \cite{GM88}) in commutative polynomial rings. More precisely, Propositions~\ref{ncMCrit},~\ref{ncLWCrit} and~\ref{ncBKCrit} correspond to criterion $M$, criterion $F$ and criterion $B_k$, respectively (c.f. \cite{GM88}, Subsection 3.4).
\end{remark}

Using the Gebauer-M\"oller criteria, we can improve the Buchberger Procedure as follows.

\begin{theorem}[(Improved Buchberger Procedure)]\label{sec3the14}
In the setting of Theorem~\ref{BuProc}, we replace step {\rm (B4)} by the following sequence of instructions.
\begin{itemize}
\item[{\rm (4a)}] Increase $s$ by one. Append $g_{s}=S'$ to the set ${G}$, and form the set of non-trivial obstructions ${\rm NTObs}(s)=\bigcup_{1\leq i\leq s}{\rm NTObs}(i,s)$.

\item[{\rm (4b)}] Remove from ${\rm NTObs}(s)$ all obstructions ${\rm o}_{i,s}(w_{i},w'_{i};u_{s},u'_{s})$ such that there exists an obstruction ${\rm o}_{j,s}(w_{j},w'_{j};v_{s},v'_{s})\in{\rm NTObs}(s)$ with the properties that there exist two words $w,w'\in\langle X\rangle$ satisfying $u_{s}=wv_{s}$, $u'_{s}=v'_{s}w'$ and $ww'\neq1$.

\item[{\rm (4c)}] Remove from ${\rm NTObs}(s)$ all obstructions ${\rm o}_{i,s}(w_{i},w'_{i};u_{s},u'_{s})$ such that there exists an obstruction ${\rm o}_{j,s}(w_{j},w'_{j};v_{s},v'_{s})\in{\rm NTObs}(s)$ with the properties that there exist two words $w,w'\in\langle X\rangle$ satisfying $u_{s}=wv_{s}$, $u'_{s}=v'_{s}w'$, and such that $i>j$, or $i=j$ and $ww'=1$  and $w_{i}>_{\sigma}w_{j}$.

%\item[{\rm (4d)}] Remove from ${\rm NTObs}(s)$ all obstructions ${\rm o}_{j,s}(u_{j},u'_{j};w_{s},w'_{s})$ such that there exists an obstruction ${\rm o}_{i,j}(w_{i},w'_{i};v_{j},v'_{j})\in B$ with the properties that there exist two words $w,w'\in\langle X\rangle$ satisfying $u_{j}=wv_{j}, u'_{j}=v'_{j}w'$, and such that ${\rm o}_{i,s}(ww_{i},w'_{i}w';w_{s},w'_{s})$ has no overlap.

\item[{\rm (4d)}] Remove from $B$ all obstructions ${\rm o}_{i,j}(w_{i},w'_{i};w_{j},w'_{j})$ such that there exist two words $w,w'\in\langle X\rangle$ satisfying $w\lw_{\sigma}(g_{s})w'=w_{j}\lw_{\sigma}(g_{j})w'_{j}$, and such that the following conditions are satisfied.
\begin{itemize}
\item[{\rm (i)}] The obstruction ${\rm o}_{i,s}(w_{i},w'_{i};w,w')$ is either an obstruction without overlap or a (proper) multiple of a non-trivial obstruction in ${\rm NTObs}(s)$.

\item[{\rm (ii)}] The obstruction ${\rm o}_{j,s}(w_{j},w'_{j};w,w')$ is either an obstruction without overlap or a (proper) multiple of a non-trivial obstruction in ${\rm NTObs}(s)$.
\end{itemize}

\item[{\rm (4e)}] Replace $B$ by $B\cup {\rm NTObs}(s)$ and continue with step {\rm (B2)}.
\end{itemize}
Then the resulting set of instructions is a procedure that enumerates a ${\sigma}$-Gr\"obner basis ${G}$ of $I$. If $I$ has a finite ${\sigma}$-Gr\"ober basis, it stops after finitely many steps and the resulting set ${G}$ is a finite ${\sigma}$-Gr\"obner basis of $I$.
\end{theorem}
\begin{proof}
This follows from Theorem~\ref{BuProc} and Propositions~\ref{ncMCrit},~\ref{ncLWCrit} and~\ref{ncBKCrit}.
\end{proof}

\section{Experiments and Conclusions}
\label{sec4}

In this section we want to present some experimental data which illustrate the performance of the Gebauer-M\"oller criteria presented in Propositions~\ref{ncMCrit},~\ref{ncLWCrit} and~\ref{ncBKCrit}. The computations are based on an implementation (using C++) in an experimental version of the ApCoCoA library (see \cite{Ap10}) by the second author.

\begin{example}\label{sec4exa1}
Consider the non-commutative polynomial ring $\mathbb{Q}\langle a,b\rangle$ equipped with the word ordering $\LLex$ on $\langle a,b\rangle$ such that $a>_{\LLex}b$. We take the list of \emph{finite generalized triangle groups} from \cite{RS02}, Theorem 2.12 and construct a list of ideals in $\mathbb{Q}\langle a,b\rangle$. For $k=1,\dots,13$ let $I_{k}=\langle G_{k}\rangle\subseteq\mathbb{Q}\langle a,b\rangle$ be the ideal generated by the following set of polynomials $G_{k}\subseteq\mathbb{Q}\langle a,b\rangle$.
\begin{eqnarray*}
G_{1}&=&\{a^2-1,b^3-1,(ababab^2ab^2)^2-1\},\\
G_{2}&=&\{a^2-1,b^3-1,(ababab^2)^3-1\},\\
G_{3}&=&\{a^3-1,b^3-1,(abab^2)^2-1\},\\
G_{4}&=&\{a^3-1,b^3-1,(aba^2b^2)^2-1\},\\
G_{5}&=&\{a^2-1,b^5-1,(abab^2)^2-1\},\\
G_{6}&=&\{a^2-1,b^5-1,(ababab^4)^2-1\}, \\
G_{7}&=&\{a^2-1,b^5-1,(abab^2ab^4)^2-1\},\\
G_{8}&=&\{a^2-1,b^4-1,(ababab^3)^2-1\},\\
G_{9}&=&\{a^2-1,b^3-1,(abab^2)^2-1\}, \\
G_{10}&=&\{a^2-1,b^3-1,(ababab^2)^2-1\},\\
G_{11}&=&\{a^2-1,b^3-1,(abababab^2)^2-1\},\\
G_{12}&=&\{a^2-1,b^3-1,(ababab^2abab^2)^2-1\},\\
G_{13}&=&\{a^2-1,b^3-1,(ababababab^2ab^2)^2-1\}.
\end{eqnarray*}
The following table lists some numbers of polynomials and obstructions treated by the Improved Buchberger Procedure given in Theorem~\ref{sec3the14}.
\begin{center}
\begin{tabular}{rccccccccc}\hline
$k$ & $\#(Gb)$ & $\!\!\#(RGb)\!\!$ & $\#(Tot)$ & $\#(Sel)$ & $\#(M)$ & $\#(F)$ & $\#(B_k)$ & $\rho$\\\hline
1 & 62 & 35 & 7032 & 248 & 6512 & 48 & 224 & 0.0353\\
2 & 133 & 96 & 31700 & 533 & 30571 & 70 & 526 & 0.0168\\
3 & 50 & 40 & 2828 & 197 & 2489 & 11 & 131 & 0.0697\\
4 & 64 & 28 & 4702 & 253 & 4185 & 46 & 218 & 0.0538\\
5 & 35 & 21 & 1580 & 115 & 1348 & 24 & 93 & 0.0728\\
6 & 199 & 164 & 51175 & 882 & 49126 & 26 & 1141 & 0.0172\\
7 & 200 & 164 & 51864 & 886 & 49818 & 17 & 1143 & 0.0171\\
8 & 53 & 37 & 3756 & 192 & 3357 & 19 & 188 & 0.0511\\
9 & 11 & 5 & 150 & 31 & 98 & 8 & 13 & 0.2067\\
10 & 22 & 15 & 741 & 74 & 605 & 18 & 44 & 0.0999\\
11 & 30 & 21 & 1573 & 116 & 1324 & 50 & 83 & 0.0737\\
12 & 97 & 70 & 16841 & 365 & 15989 & 97 & 390 & 0.0217\\
13 & 220 & 194 & 87673 & 1021 & 85136 & 153 & 1363 & 0.0116\\\hline
\end{tabular}
\end{center}
Here we used the following abbreviations.
\begin{itemize}
\item $\#(Gb)$ is the number of elements of the Gr\"obner basis returned by the procedure.

\item $\#(RGb)$ is the cardinality of the reduced Gr\"obner basis of the corresponding ideal.

\item $\#(Tot)$ is the total number of non-trivial obstructions constructed during the Buchberger Procedure.

\item $\#(Sel)$ is the number of actually selected and analysed non-trivial obstructions.

\item $\#(M)$ is the number of unnecessary non-trivial obstructions detected by the Non-Commutative Multiply Criterion given in Proposition~\ref{ncMCrit}.

\item $\#(F)$ is the number of unnecessary non-trivial obstructions detected by the Non-Commutative Leading Word Criterion given in Proposition~\ref{ncLWCrit}.

\item $\#(B_k)$ is the number of unnecessary non-trivial obstructions detected by the Non-Commutative Backward Criterion given in Proposition~\ref{ncBKCrit}.

\item $\rho=\#(Sel)/\#(Tot)$.
\end{itemize}
Note that $\#(RGb)$ is an invariant of the ideal which only depends on chosen word ordering. Other numbers in the table
rely also on the selection strategy. In our experiments we used the \emph{normal strategy} which first chooses the
obstruction whose S-polynomial has the lowest degree and then breaks ties by choosing the obstruction whose S-polynomial
has the smallest leading word with respect to the word ordering.
%In these experiments, the Non-Commutative Tail Reduction given in Proposition~\ref{ncTR} detected no of unnecessary non-trivial obstruction (see Remark~\ref{RemTR}). Apparently, the Non-Commutative Multiply Criterion and the Non-Commutative Leading Word Criterion had already detected all unnecessary obstructions in the set $\bigcup_{1\leq i\leq s}{\rm NTObs}(i,s)$ of newly constructed obstructions.
The low ratios $\rho$ in the table indicate that the non-commutative Gebauer-M\"oller criteria we obtained can detect most unnecessary obstructions
during the procedure.
\end{example}

\begin{example}\label{sec4exa2}
The following ideals {\tt braid3} and {\tt braid4} in the non-commutative polynomial ring
$\mathbb{Q}\langle x_{1},x_{2},x_{3}\rangle$ are taken from \cite{SL09}, Section~5.
More precisely, {\tt braid3} is the ideal generated by the set $\{-x_{2}x_{3}x_{1}+x_{3}x_{1}x_{3},$ $x_{2}x_{1}x_{2}-x_{3}x_{2}x_{3}$, $x_{1}x_{2}x_{1}-x_{3}x_{1}x_{2}, x_{1}^3+x_{1}x_{2}x_{3}+x_{2}^3+x_{3}^3\}$, and {\tt braid4} is the ideal generated by the set $\{-x_{2}x_{3}x_{1}+x_{3}x_{1}x_{3}, x_{2}x_{1}x_{2}-x_{3}x_{2}x_{3}, x_{1}x_{2}x_{3}-x_{3}x_{1}x_{2}, x_{1}^3+x_{1}x_{2}x_{3}+x_{2}^3+x_{3}^3\}$. These ideals are generated by sets of homogeneous generators. The following table lists the results of the computations of Gr\"obner bases truncated at degree~$11$ with respect to $\LLex$ on $\langle x_{1},x_{2},x_{3}\rangle$ such that $x_{1}>_{\LLex} x_{2}>_{\LLex}x_{3}$, via the Improved Buchberger Procedure.

\begin{center}
\begin{tabular}{lcccc}
\hline
 & $\#(Gb)$ & $\#(Tot)$ & $\#(Sel)$ & $\rho$\\\hline
braid3-11 & 729 & 292630 & 1663 & 0.0057\\
braid4-11 & 417 & 93823 & 1150 & 0.0123 \\\hline
\end{tabular}
\end{center}

The meaning of the symbols is the same as in Example~\ref{sec4exa1}.
In this experiment we also used the normal strategy. Moreover, since we compute truncated Gr\"obner bases,
we discard those obstructions whose S-polynomial have degrees larger than the degree of truncation.
Thus the ratios $\rho$ in the table are lower than the ratios in the table of Example~\ref{sec4exa1}.
Again, the non-commutative Gebauer-M\"oller criteria detect most unnecessary obstructions during the procedure.
\end{example}

The experimental data in Examples~\ref{sec4exa1} and~\ref{sec4exa2} show that the generalizations of the Gebauer-M\"oller criteria presented in Propositions~\ref{ncMCrit},~\ref{ncLWCrit} and~\ref{ncBKCrit} can successfully detect a large number of unnecessary obstructions. In fact, they apparently detect almost all unnecessary obstructions during the Buchberger Procedure.

\subsection*{Acknowledgements}

The second author is grateful to the Chinese Scholarship Council (CSC) for providing partial financial support. Both authors thank G. Studzinski for valuable discussions about non-commutative Gr\"obner bases. And both authors appreciate anonymous referees for  careful reading and useful suggestions.

\end{document}